\documentclass[a4paper,english,reqno]{amsart}

\usepackage{amssymb,amsthm, amsmath,amsfonts}
\usepackage{bbm}
\usepackage{graphicx}
\usepackage{subfigure}
\usepackage{microtype}
\usepackage{booktabs}
\usepackage{enumerate}
\usepackage{pgf}
\usepackage[latin1]{inputenc}
\usepackage{verbatim}
\usepackage{cite}
\usepackage{url}
\usepackage{xcolor}

\usepackage[cmtip,all]{xy}
\xyoption{arc}

\usepackage{pgf,tikz}
\usetikzlibrary{arrows,shapes,calc,positioning}
\usepackage{manfnt}
\usepackage{tikz-cd}
\usepackage{url}
\usepackage{microtype}

\usepackage[font=footnotesize]{caption}

\usepackage{xcolor}
\usepackage{hyperref}

\newcommand{\Z}{\mathbb{Z}}
\newcommand{\R}{\mathbb{R}}

\renewcommand{\H}{\mathbb{H}}

\newcommand{\F}{\mathcal{F}}
\newcommand{\scirc}{\raisebox{0.2ex}{$\scriptscriptstyle \circ$}}
\newcommand{\vf}[1]{\frac{\partial~}{\partial #1}}
\newcommand{\secf}[1]{\frac{\partial^2~}{\partial #1^2}}

\newcommand{\green}{\color{green}}

\newtheorem{theorem}{Theorem}

\newtheorem{proposition}{Proposition}

\theoremstyle{definition}

\newtheorem{remark}{Remark}

\numberwithin{equation}{section}

\begin{document}

\title[Free minimal actions which are not affable]%
      {Free minimal actions of solvable Lie groups which are not affable}


\author[F. Alcalde]{Fernando Alcalde Cuesta}
\address{Instituto de Matem\'aticas, Universidad de Santiago de Compostela, E-15782, Santiago de Compostela, Spain.
}
\email{fernando.alcalde@usc.es}
\thanks{This work was partially supported by Spanish MINECO Grant MTM2016-77642-C2-2-P), the European Regional Development Fund, and grant ANII-FCE-135352 from Uruguay. 
The research of the first author has been carried out despite the current administration of the University of Santiago de Compostela.
}

\author[\'A. Lozano]{\'Alvaro Lozano Rojo}
\address{Centro Universitario de la Defensa, Academia General Militar, 
         Ctra. Huesca s/n. E-50090 Zaragoza, Spain.}
\email{alozano@unizar.es}

\author[M. Mart\'{\i}nez]{Matilde Mart\'{\i}nez}
\address{Instituto de Matem\'atica y Estad\'{\i}stica Rafael Laguardia, 
         Facultad de Ingenier\'{\i}a, Universidad de la Rep\'ublica, 
         J.Herrera y Reissig 565, C.P.11300 Montevideo, Uruguay.}
\email{matildem@fing.edu.uy}

\subjclass[2010]{Primary 57R30; Secondary 37A20, 37B50.}

\date{}

\dedicatory{Dedicated to Jean Renault on his $70$th birthday}


\begin{abstract}
  We construct an uncountable family  of transversely Cantor laminations of 
  compact spaces defined by free minimal actions of solvable groups, which are 
  not affable and whose orbits are not quasi-isometric to Cayley graphs. 
\end{abstract}

\maketitle

\section{Introduction}

This paper is motivated by two different issues put forward by 
G.~Hector in~\cite{Hector}, and T.~Giordano, I.~Putnam, and C.~Skau in~\cite{GPS}, 
for which we give partial negative results.
\begin{list}{\labelitemi}{\leftmargin=18pt}

  \item[(1)] Hector's claim asserts that generic leaves of compact laminations 
    are quasi-isometric to Cayley graphs. From~\cite[Theorem~2]{Blanc2} 
    combined with~\cite[Theorem~C]{genericleaves} (see 
    also~\cite[Theorem~4]{Blanc2} for transversely Cantor laminations), we know 
    that this holds when the generic leaves have two ends. However, we will 
    prove that the claim is not true when the generic leaves have one end. 
    The question when they have a Cantor set of ends (which was 
    the case that really interested Hector)  remains open.
    \medskip

  \item[(2)] 
    Giordano, Putnam, and Skau conjectured that any minimal and free 
    continuous action of an amenable countable group on the Cantor set is orbit 
    equivalent to a Cantor minimal system (that is, a minimal $\Z$-action on 
    the Cantor set). 
    From~\cite[Theorems~4.8 and~4.16]{GPS}, this happens if and only if
    the orbital equivalence relation is \emph{affable} --namely, the union of 
    an increasing sequence of compact open equivalence subrelations, which 
    turns into an AF-equivalence relation endowed with the inductive limit 
    topology (see~\cite{GPS} and~\cite{Renault} for detailed definitions). We will 
    exhibit examples of amenable equivalence relations on the Cantor set which 
    are not affable, and therefore not orbit equivalent to a minimal 
    $\Z$-action. 
    This proves that there is no analogue to the famous Connes-Feldman-Weiss 
    theorem~\cite{CFW} in the topological setting. 

\end{list} 
\medskip 

To address these issues, we consider an uncountable family of 3-dimensional 
solvable Lie groups $Sol(a,b)$, with $a, b >0$,
which are not quasi-isometric to Cayley graphs unless $a=b$. These groups 
have been introduced and studied by A.~Eskin, D.~Fisher and K.~Whyte 
in~\cite{EF} and~\cite{EFW}. For those groups with $2^{b/a}$ being an integer, 
we construct a repetitive and aperiodic tiling $\hat{\mathcal{T}}$ inspired by 
the Penrose construction of an aperiodic tiling of the hyperbolic 
plane~\cite{Penrose}. As a byproduct, we obtain a repetitive and 
aperiodic tiling of the solvable group $Sol^3 = Sol(1,1)$. 
Theorem~\ref{thm:solenoid} states that the \emph{continuous hull} of $\hat{\mathcal{T}}$ --that is, the closure of the 
set of all its translates-- is a compact metric space $\mathcal{M}(a,b)$ 
endowed with a minimal free action of $Sol(a,b)$. This defines a transversely 
Cantor lamination $\F(a,b)$ on $\mathcal{M}(a,b)$, which satisfies:
\begin{list}{\labelitemi}{\leftmargin=18pt}
  
  \item[(i)] the leaves of $\F(a,b)$ and $\F(a',b')$ are quasi-isometric if and 
    only if $b/a = b'/a'$, 

  \item[(ii)] the leaves of $\F(a,b)$ are quasi-isometric to Cayley graphs if 
    and only if $a=b$, in which case $Sol(a,b)$ is isomorphic to the unimodular 
    solvable group $Sol^3$, 

  \item[(iii)] the lamination $\F(a,b)$ induces an equivalence relation 
    $\mathcal{R}(a,b)$ on a complete transversal homeomorphic to the Cantor set 
    which is not affable if $a \neq b$.

\end{list} 
In this foliated context, a transversely Cantor lamination is said to be 
\emph{affable} if the equivalence relation induced on some (every) complete 
transversal is affable. So property (iii) can be rephrased as the lamination 
$\F(a,b)$ is not affable  if $a\neq b$.
\medskip 

The problem of whether $\F(a,b)$ is affable involves studying it from 
an ergodic point of view, and associating to it two classes of measures which 
are well known in foliation dynamics: \emph{transverse 
invariant measures} and \emph{harmonic measures}. 
Theorem~\ref {thm:harmonic} asserts that harmonic measures for $\F(a,b)$ (having 
harmonic densities when they are locally desintegrated on flow boxes 
according to \cite{Garnett}) coincide with measures 
$Sol(a,b)$-invariant (which remain invariant when they are translated by any element of $Sol(a,b)$). 
Both kinds of measures provide 
\emph{quasi-invariant measures} on transversals, only defined up to equivalence, 
with respect to which the equivalence relation $\mathcal{R}(a,b)$ is amenable. A harmonic measure for $\F(a,b)$ is said to be \emph{completely invariant} when the transverse measure is $\mathcal{R}(a,b)$-invariant, that is, preserved by partial transformations whose graphs are contained in $\mathcal{R}(a,b)$.
Unless $a=b$, the lamination $\F(a,b)$ does not admit transverse 
invariant measures, and this is the obstruction we use to see that $\mathcal{R}(a,b)$ is 
not affable. 
\medskip 

To clarify this point, recall that Penrose's tiling has been used in \cite{Petite} to construct free minimal actions of the affine group on the Cantor set. All these actions also give a negative answer to question (2) as they do not admit transverse 
invariant measures according to \cite[Proposition 3.1]{Petite}. However, their orbits are always quasi-isometric to Cayley graphs. We extend here Petite's remark by proving in Proposition~\ref{prop:nonunimodular} that $\F(a,b)$ admits a transverse 
 invariant measure if and only if $Sol(a,b)$ is unimodular, that is, $a=b$. In fact, this result applies to (locally) free actions of Lie groups and  transformational groupoids as detailed in Appendix~A. Thus, any (locally) free minimal action of an amenable Lie group $G$ is not affable if $G$ is non unimodular, in contrast with the abelian case solved in \cite{GMPS}. However, it is still an open question if $Sol(1,1)$-solenoids are affable or not.

\section{A family of nonunimodular solvable Lie groups of dimension 3} \label{SSol(a,b)}

As established  in~\cite[Theorem 1.6]{EFW} (see also~\cite[Theorem 5.9]{EF}) 
and proved in~\cite{EFW2} and~\cite{EFW3}, there are Lie groups which are not 
quasi-isometric to any finitely generated group. Indeed, if $a,b>0$ with $a\neq b$, 
the semi-direct product $Sol(a,b) = \R^2\rtimes\R$ defined by the 
$\R$-action 
\[
  z\in\R\mapsto 
    \begin{pmatrix}
      e^{az} & 0 \\
      0 & e^{-bz} 
    \end{pmatrix}
  \in GL(2,\R)
\]
is a nonunimodular solvable group that does not admit any quasi-isometric 
finitely generated group. If $a=b$, then $Sol(a,b)$ is isomorphic to the 
unimodular Lie group $Sol^3$, which defines one of the eight Thurston 
geometries of closed 3-manifolds. Moreover, two Lie groups $Sol(a,b)$ and 
$Sol(a',b')$ are quasi-isometric if and only if $b/a=b'/a'$. 

Each solvable group $Sol(a,b)$ contains two transverse fields of hyperbolic 
planes $\H$, obtained as orbits of two transverse affine actions. In 
the next section, we shall use this idea to construct an aperiodic tiling of 
$Sol(a,b)$ in a similar way as R.~Penrose constructed an aperiodic tiling of 
$\H$~\cite{Penrose}. Identifying the hyperbolic plane
$\H=\{\, \alpha+\beta i\mid\beta>0 \,\}$ with the semi-direct product 
$\R\rtimes\R^*_+$, we can consider the natural inclusion into $\R^2\rtimes\R$ 
which sends $(\alpha,\beta)$ in $(\alpha,0,\frac{\log\beta}{a})$. This is a 
integral surface of the foliation defined by the invariant vector fields 
\[
  X = e^{az} \vf{x} \quad \text{and} \quad Z = - \vf{z},
\]
with flows 
\[
  h^+_s(x,y,z) = (x,y,z)\cdot(s,0,0) = (x+e^{az}s,y,z)
\]
and
\[
  g_t(x,y,z) = (x,y,z)\cdot(0,0,t) = (x,y,z+t)
\]
respectively. Since the Lie bracket $[X,Z] =aX$, the foliation is actually 
given by a locally free affine action. Equivalently both flows are related by
\begin{equation}
  \label{eq:B+}
  h^+_s \scirc g_t = g_t \scirc h^+_{e^{at}s}.
\end{equation}

The flow $h^+_s$ restricts to the horocycle flow of $\H$, but the geodesic 
flow is not obtained from the flow $g_t$ but from the reparametrized flow 
$g_{t/a}$. Indeed, the flow generated by $X$ and $Z$ in restriction to $\H$ are 
given by 
\[
  h^+_s(\alpha+\beta i) = \alpha+\beta s + \beta i
  \quad\text{and}\quad
  g_t(\alpha+\beta i) = \alpha + e^{at}\beta i.
\]
The left invariant Riemannian metric on $Sol(a,b)$ is  
\[
  e^{-2az}dx^2 + e^{2bz}dy^2 + dz^2,
\]
and then its restriction to $\H$ 
\[
  \frac{d\alpha^2 + (d\beta / a)^2}{\beta^2} 
    = \frac1{a^2} \frac{d\alpha^2 + d\gamma^2}{\gamma^2} 
\]
is conformally equivalent (by a homothety) to the Poincar\'e metric (up 
to the coordinate change $\gamma=\beta/a$). Thus, the inclusion of $\H$ into 
$Sol(a,b)$ sends geodesics and horocycles into orbits of $g_t$ and $h^+_s$ 
respectively.

The flow of the third left invariant vector field 
\[
  Y = e^{-bz}\vf{y}
\]
 is given by
\[
  h^-_s(x,y,z) = (x,y,z)\cdot(0,s,0) = (x,y+e^{-bz}s,z)\green, 
\]
satisfying
\begin{equation}
  \label{eq:B-}
  h^-_s \scirc g_t = g_t \scirc h^-_{e^{-bt}s}.
\end{equation}
This equality can be also deduced from the Lie bracket equality $[Y,Z] = -bY$.  
Finally, since $[X,Y]=0$, the flows $h^+_s$ and $h^-_s$ commute. 

\begin{figure}
  \begin{tikzpicture}[scale=1]
    \node at (0,0) {\includegraphics[width=\textwidth]{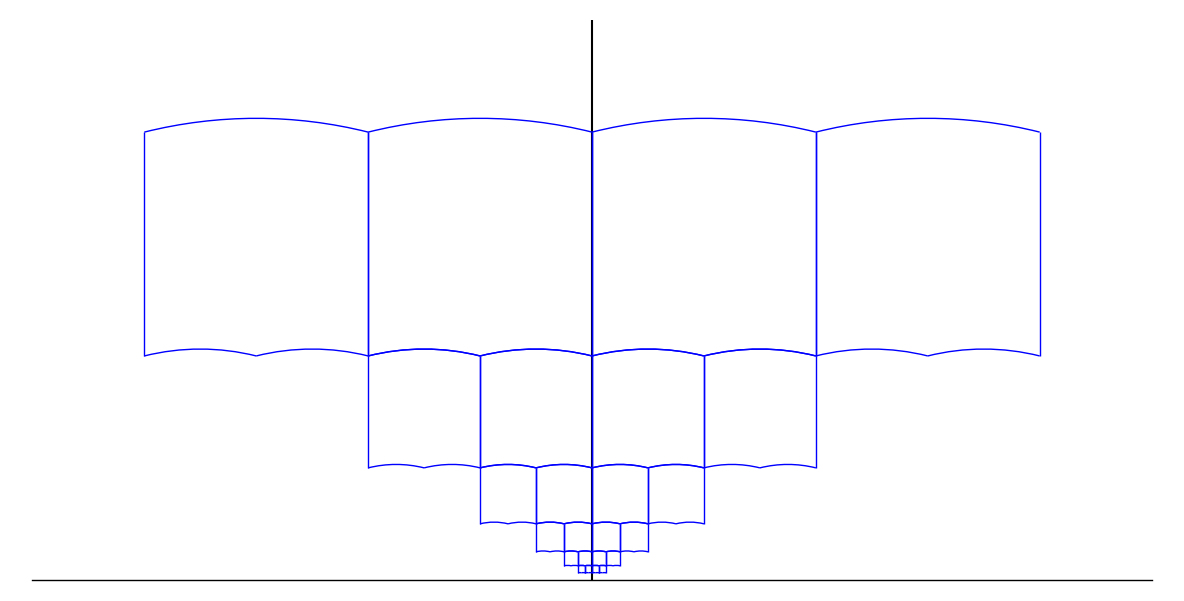}};
    \draw [->] (0.7,-0.8) -- (1.2,0.7);
    \draw[->] (0.53,-1.4) -- (0.28,-2);
    \draw[->] (0.8,-1.1) -- (1.7,-1.1);
    \node at (0.6,-1.1) {$P$}; 
    \node at (1.2,-0.9) {\scriptsize $S$}; 
    \node at (0.78,0.1) {\scriptsize $R$}; 
    \node at (0.75,-1.65) {\scriptsize $R^{-1}$}; 
  \end{tikzpicture}
  \caption{Penrose's tiling of $\H$.}
  \label{fig:Penrose}
\end{figure}
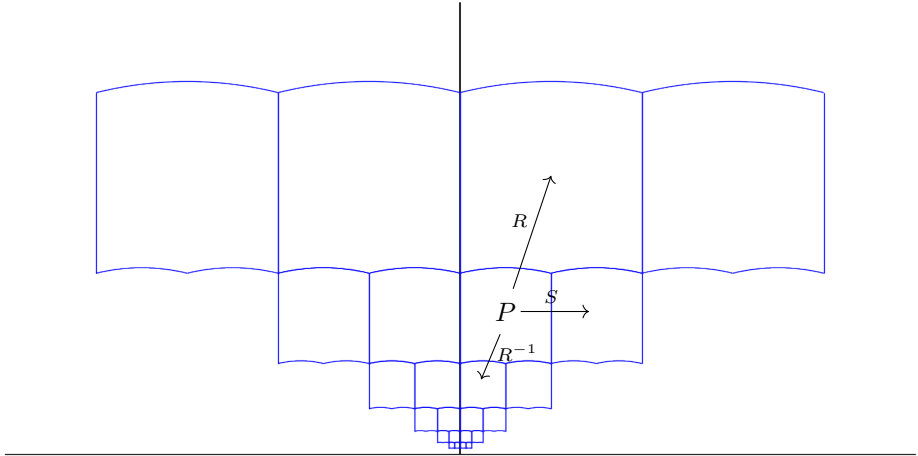

\section{The hyperbolic Penrose tiling} \label{SPenrose}
According to~\cite{Penrose}, the Poincar\'e half-plane $\H$ admits a tiling 
$\mathcal{T}$ constructed from a single tile $P$ (see 
Figure~\ref{fig:Penrose}), which is neither \emph{periodic} nor 
\emph{aperiodic}.  Let us explain the meaning of both notions (see 
also~\cite{KP}). If we consider the isometries $R$ and $S$ given by
\[
  R(Z)=2Z \quad\text{and}\quad S(Z) = Z+1
\]
for every $Z = \alpha + i\beta \in \H$, then 
\[
  \mathcal{T} =  \{\, R^i\scirc S^j(P) \mid i,j\in\Z \,\}.
\]
The tiling $\mathcal{T}$ is not \emph{periodic} since there is no cocompact 
Fuchsian group preserving $\mathcal{T}$ (and therefore having a fundamental 
domain made up of a finite number of tiles). This follows from a homological 
argument by Penrose \cite{Penrose}, see also \cite{MargulisMozes}. Nor is 
$\mathcal{T}$ \emph{aperiodic} since $\mathcal{T}$  is preserved by the 
isometry $R$ and then the group of hyperbolic isometries preserving 
$\mathcal{T}$ is not trivial. In fact, both notions can be formulated in terms 
of affine transformations instead of isometries, that is, we can replace the 
group of orientation-preserving isometries $PSL(2,\R)$ by the affine subgroup 
\[
  B^+ = \left\{\,
    \begin{pmatrix}
      \sqrt{\beta} & \alpha/\sqrt{\beta} \\
      0            & 1/\sqrt{\beta} 
    \end{pmatrix}
    \biggm|
    \alpha,\beta \in\R, \ \beta > 0
  \,\right\}
\]
acting freely and transitively on $\H$.
\medskip 

But as explained in \cite{Petite} (and detailed below), we can decorate the 
tiles $R^i(P)$ to break down this symmetry by using a repetitive sequence $\{\omega_i\}_{i\in\Z}$ of $0$'s and $1$'s (as shown in Figure~\ref{fig:Oxtoby}).
Then $\mathcal{T}$ becomes 
\emph{aperiodic}, that is, $\mathcal{T}$  is not preserved by any non-trivial 
element of $B^+$. 
We actually obtain a set of decorated prototiles, which is not longer a 
singleton but finite, allowing to construct both aperiodic and non aperiodic 
tilings. 
\medskip 

The compact metric space $\mathcal{M}(\H)$ made up of these hyperbolic tilings 
(marked with a fixed base point) is equipped with a natural right $B^+$-action 
where each tiling is translated by the inverse of each isometry in $B^+$, that is, 
\[
  \mathcal{T}\cdot g = g^{-1}(\mathcal{T})
\]
for each $g\in B^+$.  The orbital equivalence relation $\mathcal{R}$ coincides with the natural equivalence relation that consists of moving the base point of each tiling. Recall also that two tilings in $\mathcal{M}(\H)$ are \emph{close} if they agree on a large ball in $\H$ centered at the base point, up to an affine transformation close to the identity (see~\cite{Ghys} and~\cite{Petite} for details). 
\medskip 

The closure of the orbit $\mathcal{R}[\mathcal{T}]$ is a nonempty closed invariant subset of $\mathcal{M}(\H)$, called the \emph{continuous hull of $\mathcal{T}$}, which contains a nonempty minimal subset $\mathcal{M}_0$.  In fact, the tiling $\mathcal{T}$ is \emph{repetitive}. This means that for each patch $\mathcal{P}$, there is a positive number $R > 0$ such that every ball in $\H$ of radius $R$ contains a translation copy of $\mathcal{P}$. It is a general fact (see for example~\cite{KP}) for any repetitive and aperiodic tiling that the minimal set $\mathcal{M}_0$ coincides with continuous hull of $\mathcal{T}$, and that any tiling in $\mathcal{M}_0$ is also aperiodic. Then $\mathcal{M}_0$ is equipped with a minimal free affine action.

\begin{figure}
  \begin{tikzpicture}[scale=1]
    \node at (0,0) {\includegraphics[width=\textwidth]{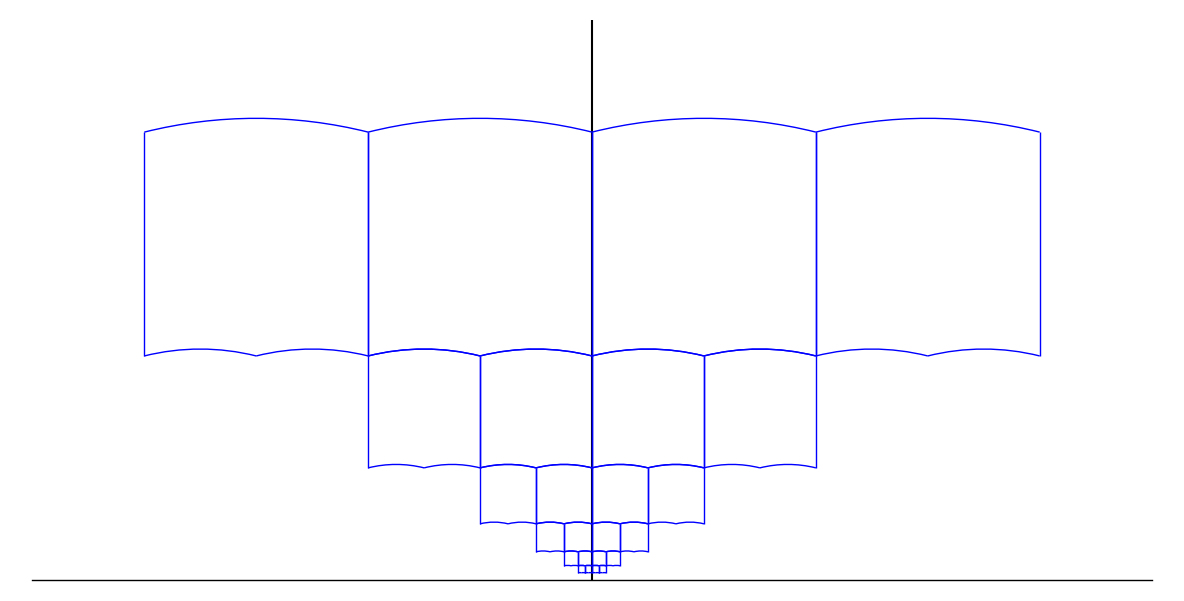}};
    \node at (0.7,-1.2)  {\footnotesize$\omega_0$}; 
    \node at (1.9,-1.2)  {\footnotesize$\omega_0$}; 
    \node at (-0.5,-1.2) {\footnotesize$\omega_0$}; 
    \node at (-1.7,-1.2) {\footnotesize$\omega_0$}; 
    \node at (1.2,0.6)  {$\omega_1$}; 
    \node at (3.6,0.6)  {$\omega_1$}; 
    \node at (-1.2,0.6) {$\omega_1$}; 
    \node at (-3.6,0.6) {$\omega_1$}; 
    \node at (0.35,-2.1)  {\tiny$\omega_{-1}$}; 
    \node at (0.95,-2.1)  {\tiny$\omega_{-1}$}; 
    \node at (-0.25,-2.1) {\tiny$\omega_{-1}$}; 
    \node at (-0.85,-2.1) {\tiny$\omega_{-1}$}; 
  \end{tikzpicture}
  \caption{Penrose's tiling decorated with a repetitive sequence.}
  \label{fig:Oxtoby}
\end{figure}

\section{Constructing an aperiodic tiling of $Sol(a,b)$}
\label{Sconst}

To obtain a similar tiling of $Sol(a,b)$, we will use the isometries
\[
  \hat{T}_s(x,y,z) = (0,s,0)\cdot(x,y,z) = (x,y+s,z)
\]
to thicken the tile $P$ into a tile 
\[
  \hat{P} = \bigcup_{s\in[0,1]} \hat T_s(P).
\]
First, we extend the isometries $R$ and $S$ into isometries of $Sol(a,b)$ 
and we see their effect on $\hat{P}$. Thus, we consider the left translations
\[
  \hat{R}(x,y,z) = (0,0,\tfrac{\log2}{a})\cdot(x,y,z) 
                 = (2x,\Delta y,z+\tfrac{\log2}{a})
\]
and 
\[
  \hat{S}(x,y,z) = (1,0,0)\cdot(x,y,z) = (x+1,y,z)
\]
where 
\begin{equation}
  \label{eq:Delta}
  \Delta = e^{-b\frac{\log2}{a}} =  2^{-b/a}.
\end{equation}
They extend $R$ and $S$ since 
\[
  \hat{R}(\alpha,0,\tfrac{\log\beta}{a}) = (2\alpha,0,\tfrac{\log2\beta}{a})
    \quad \text{and} \quad 
  \hat{S}(\alpha,0,\tfrac{\log\beta}{a}) = (\alpha+1,0,\tfrac{\log\beta}{a})
\]
for each $\alpha+i\beta \in \H$. 
Now, we have
\begin{align*}
  \hat{T}_s\scirc\hat{R}(x,y,z) &= \hat{T}_s(2x,\Delta y,z+\tfrac{\log2}{a}) \\
    & = (2x, \Delta y + s, z+\tfrac{\log2}{a}) \\
    & = \hat{R}(x, y + \Delta^{-1}s, z) \\ 
    & = \hat{R} \scirc \hat{T}_{s\Delta^{-1}}(x,y,z)
\end{align*}
for any $(x,y,z) \in Sol(a,b)$. Therefore, 
we obtain:
\begin{equation}
  \label{eq:R}
  \hat{T}_s\scirc\hat{R}^i = \hat{R}^i\scirc\hat{T}_{s\Delta^{-i}}
\end{equation}
for any $i\in\Z$. Similarly, we have
\begin{align*}
  \hat{T}_s\scirc\hat{S}(x,y,z) &= \hat{T}_s (x+1,y,z) = (x+1, y+s, z) \\
                               &= \hat{S}(x, y+s, z)
                                = \hat{S}\scirc\hat{T}_s(x,y,z)
\end{align*}
for any $(x,y,z)\in Sol(a,b)$, and therefore
\begin{equation}
  \label{eq:S}
  \hat{T}_s\scirc\hat{S}^j = \hat{S}^j\scirc\hat{T}_s
\end{equation}
for any $m\in\Z$. It follows that:


\begin{proposition} \label{thm:tiling}
  The family
  \[
    \hat{\mathcal{T}} = \{\,
      \hat{R}^i\scirc\hat{S}^j\scirc\hat{T}_k(\hat{P})
      \mid i,j,k\in\Z
    \,\}
  \]
  is a tiling of $Sol(a,b)$, by which we mean that
 $Sol(a,b)$ is the union of the tiles 
$ \hat{R}^i\scirc\hat{S}^j\scirc\hat{T}_k(\hat{P})$
and that the  intersection of any two tiles has empty interior. 
\end{proposition}

\begin{proof}
  From~\eqref{eq:R} and~\eqref{eq:S}, we deduce that
  \begin{align*}
    \hat{R}^i(\hat{T}_k(\hat{P})) &= \bigcup_{s\in[k,k+1]} 
      \hat{R}^i\scirc \hat{T}_s(P) \\
    & = \bigcup_{s\in[k,k+1]} \hat{T}_{s\Delta^i} \scirc \hat{R}^i(P) 
      = \bigcup_{s\in[k\Delta^{i}, (k+1)\Delta^{i}]}\hat{T}_s(\hat{R}^i(P))
  \end{align*}
  and 
  \begin{align*}
    \hat{S}^j(\hat{T}_k(\hat{P})) 
      &= \bigcup_{s\in[k,k+1]}  \hat{S}^j\scirc\hat{T}_s(P) \\
      &= \bigcup_{s\in[k,k+1]}  \hat{T}_s\scirc\hat{S}^j(P)
       = \bigcup_{s\in[k,k+1]} \hat{T}_s(\hat{S}^j(P))
  \end{align*}
  for $i,j\in\Z$. The real line $\R$ is covered by the intervals $[k,k+1]$ 
  without overlaps, and similarly by the intervals 
  $[k\Delta^i, (k+1)\Delta^i]$ if we apply a homothety of (fixed) ratio 
  $\Delta^i$.  Therefore, since  $\mathcal{T} = \{\, R^i\scirc S^j(P) \mid i,j\in\Z \,\}$ is 
  a tiling of $\H$ without gaps or overlaps, the family 
  $\hat{\mathcal{T}} = \{\, \hat{R}^i\scirc\hat{S}^j\scirc\hat{T}_k(\hat{P}) \mid i,j,k\in\Z \,\}$ also covers $Sol(a,b)$ without gaps or overlaps.
 \end{proof}

  Similarly to $\mathcal{T}$, the tiling $\hat{\mathcal{T}}$ is neither 
  periodic, nor aperiodic. It cannot be periodic because in that case 
  $\mathcal{T}$ would be also periodic. In fact, if $a\neq b$, this property 
  can be deduced from the non-unimodularity of $Sol(a,b)$. On the other hand, 
  $\hat{\mathcal{T}}$ remains invariant by $\hat{R}$. In fact, accordingly 
  to~\eqref{eq:R}, the group of isometries preserving $\hat{\mathcal{T}}$ is 
  reduced to the subgroup of $Sol(a,b)$ generated by $\hat{R}$.
  \medskip 
  
 Moreover, in the previous construction, the tiles of $\hat{\mathcal{T}}$ do not meet face-to-face in the $y$ direction. We will impose an additional condition, the condition that
 $$
\Delta^{-1} \in \mathbb{Z^+}.$$
This is a condition on $b/a$, allowing us to choose countably many values of  it.
  
  \begin{proposition} \label{thm:facetoface}
 If 
 $b/a = log \, n / log \, 2$ for some integer $n \geq 2$, then the tiling $\hat{\mathcal{T}}$ is face-to-face and repetitive.  Moreover, 
the prototile $\hat{P}$ admits a finite number of decorations such that $\hat{\mathcal{T}}$ also becomes aperiodic.
\end{proposition}

 \begin{proof} 
 Denote by $\tau$ the top curve of the prototile $P$ and set $\hat\tau = \bigcup_{s\in[0,1]}  \hat{T}_s(\tau)$. If 
  $\Delta^{-1}$ is a positive integer $n \geq 2$, we can divide the face $\hat\tau$ into $n$ 
  equal faces  
  $$ \bigcup_{s\in[\frac{l-1}{n}, \frac l n]} \hat{T}_s(\tau)$$ 
  for 
  $l \in\{1,2,\ldots,n\}$ as depicted in Figure~\ref{fig:that}.  In this way, the tiles of $\hat{\mathcal{T}}$ 
  meet face to face.

  Moreover, if we denote by $\omega =  \{\omega_i\}_{i\in\Z}$ the bilateral Morse 
  sequence~\cite{Auslander} or an Oxtoby sequence~\cite{Oxtoby}, we can 
  use the terms $\omega_i$ to decorate the tiles of $\hat{\mathcal{T}}$ as we decorated 
  $\mathcal{T}$ in Figure~\ref{fig:Oxtoby}. In both cases, two colors are enough to break 
  down the initial symmetry of $\mathcal{T}$ and $\hat{\mathcal{T}}$, although the continuous 
  hulls $\mathcal{M}(a,b,\omega)$ will have different ergodic properties depending 
  on the sequence used (see~\cite{Petite} for details in the case of 
  $\mathcal{T}$). Finally, since both 
 $\omega$
and $\hat{\mathcal{T}}$ are repetitive, so the decorated tiling $\hat{\mathcal{T}}(\omega)$ is.
\end{proof}

  \begin{figure}
    \centering
    \begin{tikzpicture}
      \node {\includegraphics[width=4.5cm]{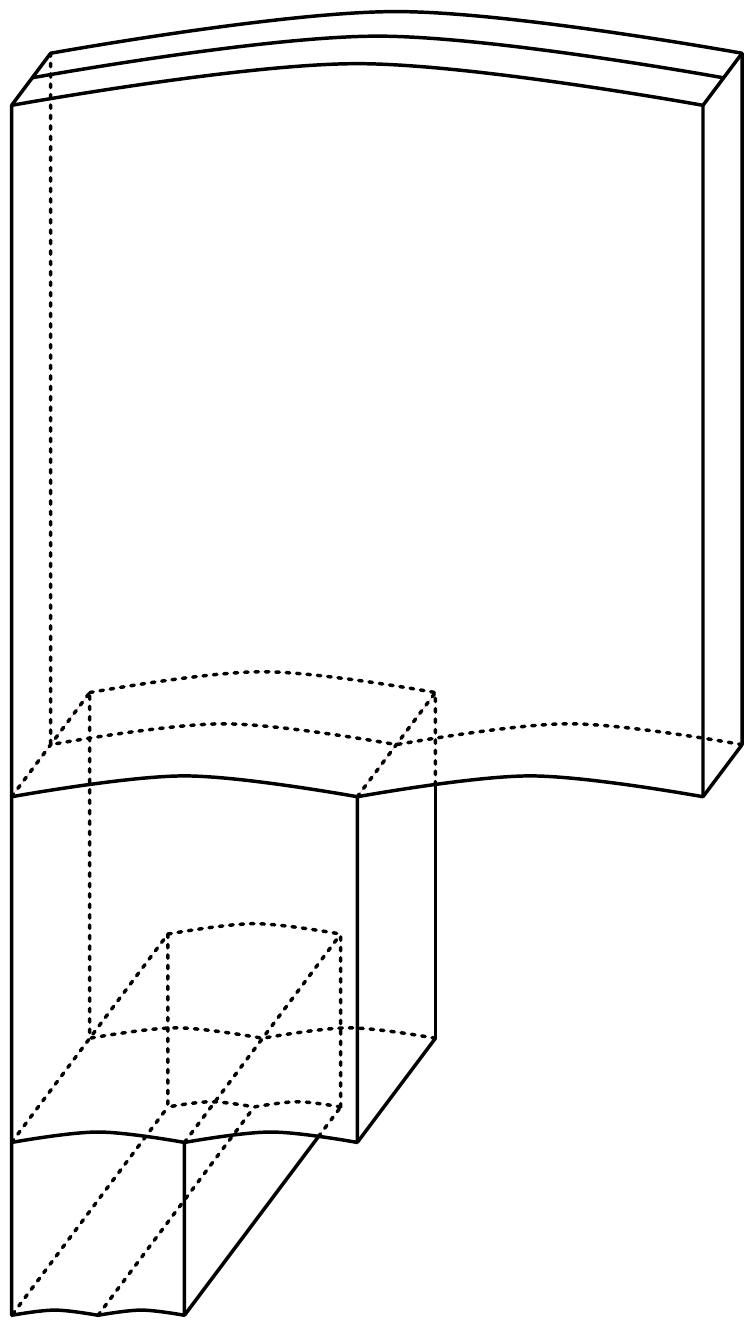}};
  \node at (-1.1,-1.25) {$\hat{P}$}; 
    \end{tikzpicture}
    \caption{The tiling $\hat{\mathcal{T}}$ when $\Delta = \frac12$.}
    \label{fig:that}
  \end{figure}
  
\begin{remark} As explained before, the groups $Sol(a,b)$ considered by Eskin, Fisher and White verify $a,b >0$.
If $a > 0$ and $b < 0$, the groups $Sol(a,b)$ are examples of Heintze groups having negative curvature. In this case, the fact that $Sol(a,b)$ is not quasi-isometric to a Cayley graph was previously proved by B. Kleiner as pointed in \cite{EF} (see also \cite{EFW} and \cite{EFW2}). Now, Proposition~\ref{thm:facetoface} remains valid if 
$\Delta \in \mathbb{Z^+}$,
or equivalenty if 
$$
-b/a = log \, n / log \, 2
$$ 
for some integer $n \geq 2$. 
\end{remark}
  

\section{A $Sol(a,b)$-solenoid} 

Consider the space of tilings of $Sol(a,b)$ constructed from the decorated 
prototiles constructed in Proposition~\ref{thm:facetoface}.
It admits an action by left translations of $Sol(a,b)$, 
which is of course an action by isometries for any left-invariant metric. Two 
tilings of $Sol(a,b)$ remain close if they agree on a large ball in $Sol(a,b)$ 
centered at the identity element, up to an isometry close to the identity map. 
The orbital equivalence relation still coincides with the natural equivalence 
relation that consists of translating the base point of each tiling. 
Restricted to the continuous hull $\mathcal{M}(a,b,\omega)$ of the decorated tiling 
 $\hat{\mathcal{T}}(\omega)$, the action of $Sol(a,b)$ is free and minimal. 
Finally, as  $\hat{\mathcal{T}}(\omega)$ satisfies the finiteness condition of~\cite[Theorem~2.2]{KP}, 
 $\mathcal{M}(a,b,\omega)$ is compact. Therefore, we get the result announced in the introduction:

\begin{theorem} \label{thm:solenoid}
  If 
$\pm b/a = log \, n / log \, 2$ for some integer $n\geq2$, the convex hull 
$\mathcal{M}(a,b,\omega)$ of  $\hat{\mathcal{T}}(\omega)$ is a nonempty compact metric space 
  that has a free minimal action of $Sol(a,b)$. With the foliation given by the 
  orbits, it has the structure of a transversely Cantor lamination.
\end{theorem}


Notice that   $\mathcal{M}(a,b,\omega)$ is a $Sol(a,b)$-solenoid in the sense 
of~\cite{BG} since tilings are translated by isometries (or equivalently the 
right $Sol(a,b)$-action on  $\mathcal{M}(a,b,\omega)$ is derived from the natural left 
action of $Sol(a,b)$ on itself) although other free minimal actions of 
$Sol(a,b)$ on compact spaces can be constructed using similar ideas. 
\medskip 

Since $Sol(a,b)$ is amenable, there is always a probability measure $\mu$ on 
 $\mathcal{M}(a,b,\omega)$ which is invariant under the action of  $Sol(a,b)$. If we use 
the Morse sequence to construct the tiling  $\hat{\mathcal{T}}(\omega)$, 
 this action is uniquely ergodic, whereas  $\mathcal{M}(a,b,\omega)$  admits many ergodic invariant 
measures when we use an Oxtoby sequence  $\omega$ to decorate the tiles of 
$\hat{\mathcal{T}}$ similarly to \cite{Petite}. 
\medskip 

If $a = b$, the Lie group $Sol(a,b)$ is unimodular, isomorphic to $Sol^3$, and 
then $\mu$ is completely invariant (see the proof of Theorem~5.2 
of~\cite{Connell-Martinez}). In fact, using the natural foliated structure of 
 $\mathcal{M}(a,b,\omega)$, we can directly prove the following result which is the 
analogue of \cite[Proposition 3.1]{Petite} in our context. A more general 
version valid for any lamination defined by a locally free action is given in 
Appendix~A.

\begin{proposition} \label{prop:nonunimodular}
  The space of tiling  $\mathcal{M}(a,b,\omega)$ admits a completely invariant measure 
  $\mu$ if and only if $a = b$. 
\end{proposition}

\begin{proof} 
  Let $\mu$ be a probability measure on  $\mathcal{M}(a,b,\omega)$ which is invariant 
  by the right $Sol(a,b)$-action. The space  $\mathcal{M}(a,b,\omega)$ is covered by 
  a finite number of  flow boxes $U\cong D\times T$ where $T$ is a clopen 
  subset of the canonical transversal (obtained by fixing base points in the 
  prototiles and homeomorphic to the Cantor set). In restriction to each flow 
  box $U$, the measure $\mu$ disintegrates into a family of probability 
  measures $\mu_t$ on the plaques $D\times\{t\}$ with respect to the 
  push-forward measure $\nu$ on $T$, that is, 
  \[
    d\mu(p,t) = d\mu_t(p)\,d\nu(t)
  \]
  for every $(p,t)\in U\cong D\times T$. Moreover, since $\mu$ is invariant under
  the right $Sol(a,b)$-action, for $\nu$-almost $t\in T$, $\mu_t$ is also 
  invariant under the right $Sol(a,b)$-action, so $\mu_t$ is the restriction of 
  the right Haar measure of $Sol(a,b)$. In fact, in our case, plaques are 
  simply tiles and changes of coordinates are obtained from transformations 
  $\hat{R}^n\scirc\hat{S}^m\scirc\hat{T}_k$ used in the construction of  
  $\hat{\mathcal{T}}$ in 
  Proposition~\ref{thm:tiling}. In other words, the changes 
  of coordinates are left translations in the group $Sol(a,b)$, and hence the 
  measure $\mu$ is completely invariant if and only the right Haar measure on 
  $Sol(a,b)$ is also left invariant. The left Haar measure on $Sol(a,b)$ is 
  given by 
  \[
    -e^{(b-a)z}dx\wedge dy\wedge dz
  \]
  while the right Haar measure is given by $-dx\wedge dy\wedge dz$. So 
  $Sol(a,b)$ is unimodular if and only if $a=b$. 
\end{proof}

More generally, any measure $Sol(a,b)$-invariant
can be interpreted as a harmonic measure, as in~\cite[Lemma 4.2]{Petite}:

\begin{proposition}
  \label{prop:harmonic}
 Any $Sol(a,b)$-invariant
probability measure $\mu$ on $\mathcal{M}(a,b,\omega)$
 is harmonic.
\end{proposition}

\begin{proof} 
  Let $\mu$ be a probability measure on  $\mathcal{M}(a,b,\omega)$ which is invariant under the 
  right $Sol(a,b)$-action. As in the proof of 
  Proposition~\ref{prop:nonunimodular}, in restriction to each flow 
  $U\cong D\times T$, the measure $\mu$ disintegrates into a family of 
  probability measures $\mu_t$ on the plaques $D\times\{t\}$ with respect to 
  the push-forward measure $\nu$ on subset $T$. Moreover, for $\nu$-almost 
  $t\in T$, the measure $\mu_t$ is induced by the right Haar measure of 
  $Sol(a,b)$. But this measure
  \[
    -dx\wedge dy\wedge dz
  \]
  is absolutely continuous with respect to the Riemannian volume
  \[
    -e^{(b-a)z}dx\wedge dy\wedge dz
  \]
  with harmonic density $e^{(a-b)z}$. Then $\mu$ is harmonic.
\end{proof}

In fact, the study of brownian motion and harmonic functions on $Sol(a,b)$ by S. Brofferio, M. Salvatori and W. Woess \cite{BSW} allows us to generalize \cite[Theorem 1.1]{Petite} to our context:

\begin{theorem} \label{thm:harmonic} 
 A probability measure $\mu$ on $\mathcal{M}(a,b,\omega)$ is harmonic if and only if it is $Sol(a,b)$-invariant.
\end{theorem} 
\begin{proof} Assume $\mu$ is a harmonic measure on  $\mathcal{M}(a,b,\omega)$. By \cite[Theorem 1]{Garnett}, in restriction to each flow box $U \cong D \times T$, the measure $\mu$ disintegrates again into a family of measures $\mu_t$ on the plaques $D \times \{t\}$ with respect to the push-forward measure $\nu$ on $T$ where each measure $\mu_t$ is absolutely continuous with respect to the Riemannian volume $dvol$ having harmonic density $h(-,t)$. More precisely, for any positive continuous function $f : \mathcal{M}(a,b,\omega) \to \R$ with support contained in $U$, we have:
$$
    \int_M f d \mu = \int_T \int_{D} \, f(x,y,z,t) h(x,y,z,t) dvol(x,y,z) \, d\nu(t)
$$
where $dvol(x,y,z) =  -e^{(b-a)z}dx\wedge dy\wedge dz$ is invariant by left translations. Thus, for each point $t \in T$, the harmonic density $h(x,y,z,t)$ defined on the plaque $D$ extends to a positive harmonic function $h(-,t)$ defined on the whole Lie group $Sol(a,b)$.
If $R_g : \mathcal{M}(a,b,\omega) \to \mathcal{M}(a,b,\omega)$ is the right translation by an element $g = (\alpha,\beta,\gamma)$ of $Sol(a,b)$, then $f \scirc R_{g^{-1}} : \mathcal{M}(a,b,\omega) \to \mathcal{M}(a,b,\omega)$ is a positive continuous function whose support in contained in $U.g = R_g(U)$. Assuming that the support of $f$ is contained in both flow boxes $U$ and $U.g$, we have: 
\begin{equation} \label{integral}
\begin{split}
  \int_M f d (R_{g^{-1}})_\ast \mu & =  \int_T \int_{D} \, f((x,y,z).g^{-1},t) h(x,y,z,t)  dvol(x,y,z)  \, d\nu(t) \\
  &  =  \int_T \int_{D} \, f(x,y,z,t)  h((x,y,z).g,t) \frac{dvol(x,y,z)}{e^{(a-b)\gamma}} \, d\nu(t). 
  \end{split}
  \end{equation}
Our aim is to prove that this integral does not depend on $g$, which enables us (by decomposing the support of $f \scirc R_{g^{-1}}$ into smaller pieces contained in the flow boxes $U$) to conclude that $(R_g)_\ast \mu = \mu$.  Indeed, as $h(x,y,z,t)$ is harmonic, the map 
$$
g = (\alpha,\beta,\gamma) \in Sol(a,b) \; \longmapsto  \;  \frac{h(x+e^{az}\alpha,y+e^{-bz}\beta,z+\gamma)}{e^{(a-b)\gamma}}
$$
is also harmonic and therefore the bounded map which sends $g = (\alpha,\beta,\gamma) \in Sol(a,b)$ onto the integral~\eqref{integral} can be written as 
$$
g = (\alpha,\beta,\gamma) \in Sol(a,b) \;   \longmapsto \;   \frac{H(\alpha,\beta,\gamma)}{e^{(a-b)\gamma}}
$$
where $H$ is a positive harmonic function on $Sol(a,b)$. But according to \cite[Corollary 6.3]{BSW}, such a harmonic function decomposes as 
$$
H(x,y,z) = H_1(x,z) + H_2(y,-z)
$$
where $H_1$ is a harmonic function on the hyperbolic plane $\H(a)$ with respect to the Riemannian metric $ds^2 = e^{-2az}dx^2 + dz^2$ and the Laplace-Beltrami operator $\Delta_1 = e^{2az} \secf{x} + \secf{z} + (b-a)\vf{z}$ and $H_2$ is the harmonic function on the 
hyperbolic plane $\H(b)$ with respect to the Riemannian metric $ds^2 = e^{2bz}dy^2 + dz^2$ and the Laplace-Beltrami operator $\Delta_2 = e^{2bz} \secf{x} + \secf{z} + (a-b)\vf{z}$. By appying \cite[Lemma 4.1]{Petite} in the case of $\H(a)$ and $\H(b)$, we deduce that $H_1(x,z)/e^{(a-b)z}$ and 
$H_2(y,-z)/e^{(a-b)z}$ are constant. We deduce that $H(x,y,z)/e^{(a-b)z}$ is constant and then the integral~\eqref{integral} does not depend on $g$. 
\end{proof} 

Consequently, if $a \neq b$, the amenable equivalence relation 
$\mathcal{R}(a,b,\omega)$ induced on the canonical transversal we described above 
cannot be affable, proving Property (iii). Indeed, as $Sol(a,b)$ is solvable, 
the lamination $\mathcal{F}(a,b,\omega)$ is amenable with respect to $\mu$. Therefore, the 
equivalence relation $\mathcal{R}(a,b,\omega)$ induced on the canonical transversal is 
amenable with respect to the quasi-invariant measure class $[\nu]$. On the 
other hand, by definition, the lamination $\mathcal{F}(a,b,\omega)$ is affable if and only if the 
equivalence relation $\mathcal{R}(a,b,\omega)$ is affable, given as the union of an 
increasing sequence of compact open equivalence subrelations. According to 
\cite[Theorem 4.8]{GPS}, any AF-equivalence relation is orbit equivalent to a 
Cantor minimal $\Z$-system. But such a minimal system always admits an 
invariant measure defining a transverse 
invariant measure on 
 $\mathcal{M}(a,b,\omega)$. By Proposition~\ref{prop:nonunimodular}, this only happens 
when $a=b$. Note however that $\mathcal{R}(a,b,\omega)$  is not Kakutani-equivalent to 
any free action of a finitely generated group as the orbits of the 
$Sol(a,b)$-action are not quasi-isometric to Cayley graphs (according to 
Property (ii) deduced from~\cite{EFW2}). Finally, Property (i) follows from the 
quasi-isometric classification of solvable groups $Sol(a,b)$, also proved 
in~\cite{EFW2}.

\section*{Appendix A. Transverse 
invariant measures for locally free actions}

Let $M$ be a compact metric space endowed with a lamination $\F$ defined by a 
right locally free action $\varphi:M\times G\to M$ of a Lie group $G$. The leaf 
passing through $x\in M$ is the orbit $x.G = \{\, x.g\mid g\in G \,\}$ 
identified with the homogenous manifold $G_x\backslash G$ where 
$G_x = \{\, g\in G \mid x.g = x \,\}$. Given any left invariant Riemannian 
metric on $G$, the space $M$ can be covered by a finite number of flow boxes 
$U_i = T_i.B$ which are obtained by translating finitely many transversals 
$T_i \subset M$ by the elements of some ball $B$ in $G$ centered at the 
identity element. In particular, holonomy transformations are defined by the 
right action of $G$ on transversals. The change of coordinates from $U_i$ to 
$U_j$ is given by left translations in $G$, that is, by local isometries 
between the plaques of $U_i$ and $U_j$. 

The leaves of $\F$ are then naturally endowed with the Riemannian volume of 
$G$, which corresponds to the left Haar measure $m_\ell$. The modular function 
$\lambda: G \to \R^\ast_+$ is given by 
\[
  \int_G f(gg_0)dm_\ell(g) = \lambda(g_0) \int_G f(g) dm_\ell(g)
\]
for any positive measurable function $f : G \to \R$. In other words, the 
measure $m_\ell$ is right invariant (and therefore $G$ is unimodular) if and 
only if $\lambda= 1$. 

\begin{theorem} \label{thm:unimodular}
  Let $(M,\F)$ be a compact lamination defined by a locally free action of a 
  Lie group $G$. If $\F$ admits a transverse
invariant measure, then 
  $G$ is unimodular.
\end{theorem}

\begin{proof}
  Let $\nu$ be a transverse invariant measure, and let $\mu$ be the completely 
  invariant measure on $M$ obtained by integrating the left Haar measure 
  $m_\ell$ with respect to $\nu$. If $U = T.B$ is a flow box and $f:M\to\R$ is 
  a positive continuous function with support contained in $U$, then 
  \[
    \int_M f d \mu = \int_T \int_{B} f(x.g) \, dm_\ell(g) \, d\nu(x).
  \]
  If $R_{g_0} : M \to M$ is the right translation by an element $g_0 \in G$, 
  then $f \scirc R_{g_0} : M \to \R$ is a positive continuous function whose 
  support is contained in the flow box $U.{g_0^{-1}} = T.(Bg_0^{-1})$.  Then 
  \begin{align*}
    \int_M f \scirc R_{g_0} d \mu &= \int_T \int_{Bg_0^{-1}} f \scirc R_{g_0} (x.g) \, dm_\ell(g) \, d\nu(x) \\
&= \int_T \int_{Bg_0^{-1}} f (x.(gg_0)) \, dm_\ell(g) \, d\nu(x) \\
 &= \int_T  \lambda(g_0) \int_{B} f (x.g) \, dm_\ell(g) \, d\nu(x) = \lambda(g_0) \int_M f d \mu 
  \end{align*}
  and hence 
  \[
    (R_{g_0})_\ast \mu =  \lambda(g_0) \mu. 
  \]
  Since both  $\mu$ and $(R_{g_0})_\ast \mu$ are probability measures, this 
  implies that $\lambda(g_0)=1$. Therefore, $G$ is unimodular. 
\end{proof}

 As  pointed out in the introduction, this result also fits the theory of measured groupoids. Indeed, as shown by A. Connes in \cite{Connes}, see also \cite{Renaultbook}, \emph{transverse measures} are essentially the same as \emph{quasi-invariant measures} in Mackey's theory of virtual groups. In particular, an \emph{invariant
transverse measure} is a quasi-invariant measure of module $1$. The terminology for such a measure  in the foliated context is \emph{transverse holonomy-invariant measure}, usually shorten as \emph{transverse invariant measure}. In the~case of the transformational
groupoid $M \rtimes G$, a quasi-invariant measure in the sense of Mackey is exactly a measure $\mu$ on M
which is quasi-invariant under the action of $G$. The same computation
we just did in the proof of Theorem~\ref{thm:unimodular} (see also \cite[Corollaire~I.7]{Connes} and \cite[Section~I.3.21]{Renaultbook}) shows that the module $\Delta$ of $\mu$ satisfies 
$$
\Delta(x,g)  D(x,g) \lambda(g) = 1
$$
where $\lambda$ is the modular function of $G$ defined above (inverse of the usual definition) and $D$ is the Radon-Nikodym cocycle defined as
$$
(R_{g^{-1}})_\ast \mu = D(\cdot,g) \mu.
$$
If $\mu$ is a probability measure of module $\Delta = 1$, this implies $\lambda =1$. Thus, if $G$ is not unimodular, there are no finite invariant transverse measures. The authors would like to thank the referee for this remark, which is reproduced almost verbatim, and for a careful reading of the manuscript.


\medskip 


\begin{thebibliography}{10}

\bibitem{Auslander}
Joseph Auslander, \emph{Minimal flows and their extensions}, North-Holland
  Mathematics Studies, vol. 153, North-Holland Publishing Co., Amsterdam, 1988.
  \MR{956049 (89m:54050)}

\bibitem{BG}
Riccardo Benedetti and Jean-Marc Gambaudo, \emph{On the dynamics of {$\Bbb
  G$}-solenoids. {A}pplications to {D}elone sets}, Ergodic Theory Dynam.
  Systems \textbf{23} (2003), no.~3, 673--691. \MR{1992658}

\bibitem{Blanc2}
Emmanuel Blanc, \emph{Laminations minimales r\'esiduellement \`a 2 bouts},
  Comment. Math. Helv. \textbf{78} (2003), no.~4, 845--864. \MR{2016699}

\bibitem{BSW}
Sara Brofferio, Maura Salvatori, and Wolfgang Woess, \emph{Brownian motion and
  harmonic functions on {${\rm Sol}(p,q)$}}, Int. Math. Res. Not. IMRN (2012),
  no.~22, 5182--5218. \MR{2997053}

\bibitem{Connell-Martinez}
Chris Connell and Matilde Mart\'{\i}nez, \emph{Harmonic and invariant measures
  on foliated spaces}, Trans. Amer. Math. Soc. \textbf{369} (2017), no.~7,
  4931--4951. \MR{3632555}

\bibitem{Connes}
Alain Connes, \emph{Sur la th\'{e}orie non commutative de l'int\'{e}gration},
  Alg\`ebres d'op\'{e}rateurs ({S}\'{e}m., {L}es {P}lans-sur-{B}ex, 1978),
  Lecture Notes in Math., vol. 725, Springer, Berlin, 1979, pp.~19--143.
  \MR{548112}

\bibitem{CFW}
Alain Connes, Jacob Feldman, and Benjamin Weiss, \emph{An amenable equivalence
  relation is generated by a single transformation}, Ergodic Theory Dynam.
  Systems \textbf{1} (1981), no.~4, 431--450 (1982). \MR{662736}

\bibitem{EF}
Alex Eskin and David Fisher, \emph{Quasi-isometric rigidity of solvable
  groups}, Proceedings of the {I}nternational {C}ongress of {M}athematicians.
  {V}olume {III}, Hindustan Book Agency, New Delhi, 2010, pp.~1185--1208.
  \MR{2827837}

\bibitem{EFW}
Alex Eskin, David Fisher, and Kevin Whyte, \emph{Quasi-isometries and rigidity
  of solvable groups}, Pure Appl. Math. Q. \textbf{3} (2007), no.~4, Special
  Issue: In honor of Grigory Margulis. Part 1, 927--947. \MR{2402598}

\bibitem{EFW2}
\bysame, \emph{Coarse differentiation of quasi-isometries {I}: {S}paces not
  quasi-isometric to {C}ayley graphs}, Ann. of Math. (2) \textbf{176} (2012),
  no.~1, 221--260. \MR{2925383}

\bibitem{EFW3}
\bysame, \emph{Coarse differentiation of quasi-isometries {II}: {R}igidity for
  {S}ol and lamplighter groups}, Ann. of Math. (2) \textbf{177} (2013), no.~3,
  869--910. \MR{3034290}

\bibitem{Garnett}
Lucy Garnett, \emph{Foliations, the ergodic theorem and {B}rownian motion}, J.
  Funct. Anal. \textbf{51} (1983), no.~3, 285--311. \MR{703080}

\bibitem{genericleaves}
\'{E}tienne Ghys, \emph{Topologie des feuilles g\'{e}n\'{e}riques}, Ann. of
  Math. (2) \textbf{141} (1995), no.~2, 387--422. \MR{1324140}

\bibitem{Ghys}
\'Etienne Ghys, \emph{Laminations par surfaces de {R}iemann}, Dynamique et
  g\'eom\'etrie complexes ({L}yon, 1997), Panor. Synth\`eses, vol.~8, Soc.
  Math. France, Paris, 1999, pp.~ix, xi, 49--95. \MR{1760843}

\bibitem{GMPS}
Thierry Giordano, Hiroki Matui, Ian~F. Putnam, and Christian~F. Skau,
  \emph{Orbit equivalence for {C}antor minimal {$\Bbb Z^d$}-systems}, Invent.
  Math. \textbf{179} (2010), no.~1, 119--158. \MR{2563761}

\bibitem{GPS}
Thierry Giordano, Ian Putnam, and Christian Skau, \emph{Affable equivalence
  relations and orbit structure of {C}antor dynamical systems}, Ergodic Theory
  Dynam. Systems \textbf{24} (2004), no.~2, 441--475. \MR{2054051}

\bibitem{Hector}
Gilbert Hector, Personal communication, 2016.

\bibitem{KP}
Johannes Kellendonk and Ian~F. Putnam, \emph{Tilings, {$C^\ast$}-algebras and
  {K}-theory}, Directions in Mathematical Quasicrystals, CRM Monograph Series,
  vol.~13, Amer. Math. Soc., Providence, RI, 2000, pp.~177--206. \MR{1798993}

\bibitem{MargulisMozes}
Grigiori\v{\i}~A. Margulis and Shahar Mozes, \emph{Aperiodic tilings of the
  hyperbolic plane by convex polygons}, Israel J. Math. \textbf{107} (1998),
  319--325. \MR{1658579}

\bibitem{Oxtoby}
John~C. Oxtoby, \emph{Ergodic sets}, Bull. Amer. Math. Soc. \textbf{58} (1952),
  116--136. \MR{0047262}

\bibitem{Penrose}
Roger Penrose, \emph{Pentaplexity: a class of nonperiodic tilings of the
  plane}, Math. Intelligencer \textbf{2} (1979/80), no.~1, 32--37. \MR{558670}

\bibitem{Petite}
Samuel Petite, \emph{On invariant measures of finite affine type tilings},
  Ergodic Theory Dynam. Systems \textbf{26} (2006), no.~4, 1159--1176.
  \MR{2247636}

\bibitem{Renaultbook}
Jean Renault, \emph{A groupoid approach to {$C^{\ast} $}-algebras}, Lecture
  Notes in Mathematics, vol. 793, Springer, Berlin, 1980. \MR{584266}

\bibitem{Renault}
\bysame, \emph{A{F} equivalence relations and their cocycles}, Operator
  algebras and mathematical physics ({C}onstan\c ta, 2001), Theta, Bucharest,
  2003, pp.~365--377. \MR{2018241}

\end{thebibliography}
\end{document}